\documentclass[10pt,amssymb,amsfonts,psfig]{amsart}
\usepackage{epsfig,amsfonts,amssymb,color,amsmath,amscd}

\newtheorem{thm}{Theorem}[section]

\newtheorem{lem}[thm]{Lemma}
\newtheorem{prop}[thm]{Proposition}

\theoremstyle{remark}
\newtheorem{rem}{Remark}[section]

\newtheorem{example}{Example}[section]

\theoremstyle{definition}
\newtheorem{defn}[thm]{Definition}

\newtheorem*{LY Theorem}{Lee-Yang Theorem}

\numberwithin{equation}{section}
\numberwithin{figure}{section}

\font\nt=cmr7

\def\note#1
{\marginpar
{\nt $\leftarrow$
\par
\hfuzz=20pt \hbadness=9000 \hyphenpenalty=-100 \exhyphenpenalty=-100
\pretolerance=-1 \tolerance=9999 \doublehyphendemerits=-100000
\finalhyphendemerits=-100000 \baselineskip=6pt
#1}\hfuzz=1pt}

\newcommand{\Hom}{\operatorname{Hom}}

\newcommand{\spec}{\operatorname{spec}}

\newcommand{\Z}{{\Bbb Z}}

\renewcommand{\j}{{\bar j}}

\def\B0{{\mathbf{0}}}
\def\n{{\mathbf{n}}}
\def\m{{\mathbf{n}}}

\newcommand{\Det}{\operatorname{Det}}

\catcode`\@=12

\def\Empty{}
\newcommand\oplabel[1]{
  \def\OpArg{#1} \ifx \OpArg\Empty {} \else
  	\label{#1}
  \fi}
		
\newcommand{\comm}[1]{}
\newcommand{\comment}[1]{}

\begin{document}

\bigskip\bigskip

\title[Periods ]{Periodic orbits of a dynamical
 system related to a knot}

\author {Lilya Lyubich}
\date{\today}

 \begin{abstract} 
   Following \cite{SiWi 2} we  consider a knot group $G$,
 its commutator subgroup
 $K=[G,G]$, a finite  group $\Sigma$ and the space $\Hom  
 (K,\Sigma )$ 
of all representations \(\rho : K \rightarrow \Sigma\),
 endowed with the weak topology. We choose a meridian \( x \in G\) of the knot
 and consider the  homeomorphism $ \sigma_{x}$ 
 of  $ \Hom (K, \Sigma)$ onto itself:
\( \sigma_{x}\rho(a) = \rho(xax^{-1})   \  \forall   a\in K,\   \rho\in
 \Hom (K, \Sigma )
 \).
As proven in \cite{SiWi 1}, the dynamical system
 \(( \Hom (K,\Sigma), \sigma_{x} )\) is 
a shift of finite type. In the case when $\Sigma$ is abelian,
 $ \Hom (K,\Sigma)$ is 
finite.

In this paper we calculate the periods of orbits
 of $( \Hom  (K,\Z /p), \sigma_{x})$  where $p$ is prime
 in terms of the roots of the
 Alexander polynomial of the knot. In the case of two-bridge knots we give
a complete description of the set of periods. 

\end{abstract}

\setcounter{tocdepth}{1}
 
\maketitle
\tableofcontents

\section{Introduction }
Let $G$ be a knot group, $\kappa :G \rightarrow \Z $ be an epimorphism, and 
 let $x \in G $  be a distinguished element such that $ \kappa (x) = 1  $.
A triple $(G, \kappa , x )$ is a particular case of an augmented group
system, introduced by D.~ Silver in \cite{Si}.

Denote by $K$ the commutator subgroup of $G$,\; $K=[G,G].$ Then
 $  K= \ker \kappa$. $\; \;$ The group $ K$ may or may not be finitely
 generated, but it has a finite
 $\Z $-dynamic presentation (see for example \cite{SiWi 1}).

Let $\Sigma$ be a finite group. Let us consider a dynamical system consisting
of the set $\Hom (K, \Sigma )$ of all representations 
$\rho : K \rightarrow
 \Sigma $ endowed with the weak topology, together with the homeomorphism 
$\sigma _{x}$ (the shift map): 
\[
\sigma _{x} \; : \Hom (K, \Sigma )   \rightarrow \Hom (K, \Sigma ); \;\; 
\sigma _{x} \rho(a) = \rho (x a x^{-1}) \;\;\; \forall a \in K,\;
\rho \in \Hom (K, \Sigma ).\] 
D.~ Silver and S.~ Williams have proved in \cite{SiWi 1} that           
the dynamical system 

\noindent  
$(\Hom (K, \Sigma ) , \sigma _{x})$ is a shift of finite type,
and it can be completely described by a finite directed graph, so that
to each representation $\rho \in \Hom (K, \Sigma ) \; $ corresponds
 a bi-infinite
path in the graph.
 We consider a special case when $\Sigma = \Z /p, \; p $  is a prime.
In this case the dynamical system is particularly simple: it is finite
\cite{SiWi 2}. So the graph consists of several cycles. Our goal is 
to calculate the periods of the cycles.
 
 In section 2 we briefly describe the $\Z $-dynamic presentation of the 
commutator subgroup $K$.
In section 3  we remind  the reader the definition of the Alexander
matrix of a knot obtained from a finite presentation of the knot group.
 In section 4
we describe  $(\Hom (K, \Z /p ) , \sigma _{x})$ as a shift in the space
of sequences satisfying a recurrent equation that is closely related 
to the Alexander matrix of the knot.
 In section 5  we prove that the shift
$ \sigma _{x}$ has roughly the same spectral structure as an Alexander 
matrix  of the knot : if we exclude the eigenvectors with eigenvalue $0$,
there is one-to-one correspondence between sequences of adjoint vectors
for eigenvectors of $ \sigma _{x}$ and those of the pencil
$B-tA$, where $B-tA$ is the Alexander matrix for the Wirtinger 
presentation of the knot group.
In section 6 we review how to reconstruct the normal form of the 
Alexander matrix
 from its 
 invariant polynomials. In section 7 we calculate
 the least common multiple of periods of orbits
of  $(\Hom (K, \Z /p ) , \sigma _{x}).$ In section 8 we find the set of periods
 of orbits of the extended dynamical system  
$(\Hom (K, F ) , \sigma _{x})$, where 
$F$ is the splitting field of Alexander polynomial of the knot over $\Z /p$.
In section 9 we find the set of periods of  $(\Hom (K, \Z /p ) , \sigma _{x})$
in case of
 two-bridge knots.  Section 10 contains several examples.

I am very grateful to Kunio Murasugi for suggesting   this problem to me  
and for reading the manuscript,
to Yuri~ I.~ Lyubich, Dan Silver and Susan Williams for
useful discussions and references, and to Misha Lyubich
for his help.

\section{$\Z $- dynamic presentation of the commutator subgroup of 
a knot group}

Let $\mathcal{K}$ be a knot,\  $G$ the knot group   with distinguished
meridian $x$ and a presentation
\begin{equation}\label{pr0}
(x,x_{1},\ldots ,x_{\n}|\tilde{r}_{1},\ldots ,\tilde{r}_{\m}).
\end{equation}
There is a unique epimorphism $\kappa : G \rightarrow Z,\ $ such that
$\ \kappa(x)=1$.
Then $$ker\, \kappa\, =\,[G,G]\, =\,K$$ is the commutator subgroup of $G$.
 Clearly
 $\tilde{r}_{i} \in K$ (since $\kappa(\tilde {r}_{i})\,=\,0,\ i=1,\ldots ,\m )$.
Take $a_{i}=x_{i}x^{-\kappa(x_{i})}$. Then $ \kappa(a_{i})=0$, so
$a_{i}\in K,\  i=1,\ldots ,\n $,\ and we can rewrite presentation (\ref {pr0}) 
in the
 alphabet $(x,a_{1},\ldots , a_{\n})$:
\begin{equation}\label{presG}
G\cong (x,a_{1},\ldots ,a_{\n}|r_{1},\ldots ,r_{\m}),\  \mbox{where}\ 
                  r_{i}\in K,\  i=1,\ldots ,\n .
\end{equation}
 The total power of $x$\  in each $r_{i}$ is 0,\ and $r_{i}$ can
be written as \[x^{j_{1}}a_{i_{1}}x^{-j_{1}}x^{j_{2}}a_{i_{2}}x^{-j_{2}}\cdot 
\ldots 
\cdot x^{j_{s}}a_{i_{s}}x^{-j_{s}}.\]
Let us consider an automorphism $ \sigma _{x}   : 
 K \rightarrow K,\; \sigma _{x}(a)=
xax^{-1}\;\; \forall a \in K$.\; $K$ has a $\Z$-dynamic presentation: \;$K \cong
(a_{1j},\ldots ,a_{\n j}|r_{1j},\ldots ,r_{\m j})$,\;$ j\in \Z$, where
\[ a_{i 0}=a_{i},\;
a_{ij}\,= \,\sigma _{x}^{j}a_{i0}\,= \,x^{j}a_{i0}x^{-j},\;
  r_{i0}=r_{i},\;r_{ij} = \,\sigma_{x}^{j}r_{i0} \]
(Note that 
$ r_{ij}$ is obtained
from $r_{i0}$ by adding $j$ to the second subindex of each
 $a_{\alpha \beta}$ occurring in 
$r_{i})$. 

Shifting   if needed the second subindex: $\;j:=j+t $ in relators
 and generators, we can assume
without loss of generality that for each $i=1,\ldots ,\n $, minimal $j$ 
such that
 $a_{ij} $ occurs in
$r_{i0},\ldots ,r_{\m 0}$ is zero. Denote maximal $j$ such that $a_{ij} $ occurs
 in $r_{10},\ldots ,r_{\m 0}$\, by $M$.

Later we will use  the Wirtinger presentation of the knot group G, which is
the presentation of the form
$(x,x_{1} \ldots x_{\n}|r_{1}, \ldots r_{\n})$
where 
$x, x_{1},\ldots x_{\n} $ are all meridians,
 and to each crossing corresponds a relator
 of the form $ u v u^{-1} w^{-1} $ with $ u,v,w $ being generators.%
 \footnote  
 {Here  $\n +1$ is the number of crossings in a knot diagram.
 One relator is skipped since it is a consequence of the others.}

\section{Alexander matrix for a presentation of the knot group}

 Let $\gamma : G \rightarrow \{t^{j}\}_{j \in Z} $ be an epimorphism of
the knot group onto a multiplicative infinite cyclic group : $\gamma(x) = t ,
\, \gamma(a) = 1 \; \forall a \in K$. Remind (see \cite{CrF}, \cite{M})
 that Alexander matrix 
for presentation
 (\ref {presG})
is polynomial matrix $\m $ by $\n $,
\footnote{We delete the first column $\gamma (\frac{\partial r_{i}}{\partial x}
)$,
obtaining an equivalent matrix, i.e., a  matrix with the same 
invariant polynomials (see Section 8 of this paper).}
 $\mathcal{A}(t)$, with entries 
\[\mathcal {A}_{ik}(t) =  \gamma (\frac{\partial r_{i}}{\partial a_{k}}),\]
 where
$\frac{\partial}{\partial a_{k}}$ are the free derivatives in the group 
ring of $G$ that we denote $\Z (G)$ . We calculate 
the free derivatives using the rules:
\[\frac{\partial}{\partial a_{k}}(\nu_{1}+\nu_{2})=
\frac{\partial}{\partial a_{k}}\nu_{1} +\frac{\partial}{\partial a_{k}}\nu_{2}\]
\[\frac{\partial}{\partial a_{k}}(\nu_{1}\nu_{2})=
\frac{\partial}{\partial a_{k}}\nu_{1}+\nu_{1} \frac{\partial}{\partial a_{k}}
\nu_{2} \;\;\;\forall\nu_{1}, \nu_{2} \in Z(G)\]
Note that $$\frac{\partial a_{i}}{\partial a_{k}} = \delta _{ik}$$
 \[\frac{\partial}{\partial a} a^{n}
 = 1 + a + a^2 + \ldots +a^{n-1}, \;\;\]
 \[\frac{\partial}{\partial a} a^{-n} =
 -a^{-1} \! - a^{-2} \! - \ldots - a^{-n},\,
 \mbox{for} \; n > 0 .\]
 So when we differentiate the word \[r_{i} = w_{1}x^{j_{1}}a_{k}^{c_{1}}x^{-j_{1}}
w_{2}x^{j_{2}}a_{k}^{c_{2}}x^{-j_{2}}\cdots x^{j_{s}}a_{k}^{c_{s}}x^{-j_{s}}w_{s+1}
= w_{1}a^{c_{1}}_{kj_{1}}w_{2}a^{c_{2}}_{kj_{2}} \cdots  w_{s+1}\] where $ w_{i}$ 
are the words consisting of blocks $x^{\beta}a_{\alpha}x^{-\beta}$
 with $\alpha \neq k$, and  then apply $\gamma \; , $ 
we get \[\gamma (\frac {\partial r_{i}}{\partial a_{k}})
 = c_{1}t^{j_{1}} + c_{2}t^{j_{2}} + \ldots +  c_{s}t^{j_{s}} 
= \sum_{j=1}^{M} A_{ik}^{j} t ^{j}\]
( since $ \;  0\leq j_{1}, \ldots j_{s} \leq M),\;$
where $A_{ik}^{j}$ is equal to the total power of $a_{kj}$ in $r_{i0}.$
 The Alexander matrix for presentation (\ref{presG}) is then 
\[\mathcal{A}(t)=
\left ( \gamma (\frac {\partial r_{i}}{\partial a_{k}}) \right )_{\begin{array}
{l}
i=1, \ldots ,\m ,\,
\\ k = 1, \ldots ,\n 
\end{array}} \; \;\; =\; A^{0} + A^{1}t+ \ldots + A^{M}t^{M} \]

\begin{rem}
It is easy to check that Alexander matrices obtained from presentations (\ref
{pr0}) and (\ref{presG}) coincide:
\[\gamma(\frac {\partial r_{i}}{\partial x_{k}}) =
 \gamma(\frac {\partial r_{i}}{\partial a_{k}}).\]
\end{rem}

From now on we will consider the Alexander matrix over the group ring 
$(\Z /p)(G)$
 instead of $\Z (G)$, so that the entries of $A^j ,\, j=1,\ldots ,M ,\;$
 belong to 
$\Z /p.$

\section{Spectral structure of the shift}

From \cite{SiWi 1} we know that the dynamical system $$(\Hom (K, Z/p ),   
\sigma _{x} ),\;\;
\sigma _{x}  (\rho) (a) = \rho( xax^{-1})   \;\;\; \forall a\in K,\; \rho \in
\Hom (K, Z/p )$$
is a shift of finite type. In the following proposition we show that it is
isomorphic to a subshift of a full shift satisfying a linear recurrent equation.

\begin{thm}\label{theor1}
The dynamical system $( \Hom (K, Z/p ), \sigma _{x}  )$
 is isomorphic to the space of bi-infinite sequences
$(\ldots Y_{-1}Y_{0}Y_{1}\ldots ),\;\;Y_{j}\in (Z/p)^{\n}$,
with the left shift in the space of sequences,
 that we denote by $\sigma $ , satisfying
\begin{equation}\label{rec00}
A^{0}Y_{j}+A^{1}Y_{1+j}+\cdots +A^{M}Y_{M+j}=0, \;\;\;j \in Z ,
\end{equation}
where $A^{0}, A^{1} \ldots A^{M}$ are the matrices $\m \times \n $  \:
 and $A^{0} +A^{1}t + \ldots A^{M}t^M$ is an Alexander matrix for $G$
\footnote{By definition, the Alexander matrix for any presentation of $G$ 
is an Alexander matrix for the group $G$.}. 
\end{thm}

\begin{proof}

Let $\tilde{K} = K/[K,K].\;\;$ The abelian group $\tilde{K}$ has a presentation 
\begin{equation}\label{}
(a_{1j},\ldots , a_{\n j}| R_{1j}, \ldots ,R_{\m j}).
\end{equation}
\noindent
(We keep the same notations $a_{ij}$ for generators, and rewrite relations in
additive form).The space \(\Hom (K, \Z /p) \cong \, \Hom (\tilde
{K}, \Z /p)\) and is isomorphic
to the space of bi-infinite sequences of vectors $(\ldots Y_{-1}, Y_{0}, Y_{1},
\ldots) , \; Y_{j} \in (\Z /p)^{\n}$ ,\; via the isomorphism 
\[Y_{j} =
\left (
\begin{array}{c}
\rho(a_{1j}) \\
\vdots \\\rho(a_{\n j})
\end{array}
\right ) \; , \;\;j \in \Z ,\]
satisfying the system of recurrent equations, imposed by relations  
of $\tilde{K}$
\begin{equation}\label{ro}
\left \{ 
\begin {array}{c}
\rho(R_{1j})=0\\
\vdots\\\rho(R_{\m j})=0
\end{array}
\right. \mbox{ which is equivalent to};\;
\left \{ 
\begin{array}{c}
(\sigma _{x}^{j}\rho)(R_{10})=0\\
\vdots\\
(\sigma _{x}^{j}\rho)(R_{\n 0})=0
\end{array}
\right.  
\end{equation}
The system
\begin{equation}\label{ro0}
\left \{
\begin {array}{c}
\rho(R_{10})=0\\
\vdots\\\rho(R_{\m 0})=0
\end{array}
\right .    
\end {equation}
can be written in matrix form:
$$A^{0}Y_{0}+A^{1}Y_{1}+\cdots +A^{M}Y_{M}=0 ,$$ 
while the system (\ref{ro}) is equivalent to
$$A^{0}Y_{j}+A^{1}Y_{1+j}+\cdots +A^{M}Y_{M+j}=0, \;\;\;j \in \Z ,$$ 
where $A^{j}_{ik}$ is the coefficient at $a_{kj}$ in $R_{i0}$,
equal to the total degree of $a_{kj}$ in $r_{i0}$. 
So the matrices $A^{0}$,$A^{1}$,\ldots  $A^{M}$     are the same as we had
in previous section, i.e, $A^{0}+A^{1}t + \cdots +A^{M}t^{M}$
 is the Alexander matrix for
the presentations (\ref{pr0}) and (\ref{presG}).
\end{proof}
 
If we start with the Wirtinger presentation for $G$, we will get the Alexander
matrix  $\mathcal{A}(t) $ with entries $1-t ,\, t ,\,-1.$
We'll write it in the form $\mathcal{A}(t)= B-tA ,$ where $A$ and $B$ 
are the matrices $\n \times \n $
  consisting of $\;0, \, \pm 1. $
Equation (\ref{rec00}) then  takes form:
\begin{equation}\label{rec0}
BY_j-AY_{j+1}=0,\;\;\;j \in \Z.
\end{equation}
   
 It follows  from \cite{SiWi 2} that the space
 $\mathcal{V}$ 
of  bi-infinite sequences satisfying (\ref{rec0}) is finite; and the projection
$ \pi _{0} : (\ldots  Y_{-1}, Y_{0}, Y_{1} \ldots) \mapsto  Y_{0}$  is
injective. It means that $ \sigma$ and so $\sigma _{x}$ 
 is conjugate to an operator
 $T : V \rightarrow V;\;\;\; T(Y_{0})=Y_{1}$\;\;\;
$ Y_{0} \in V \subset (Z/p)^{\n} $  where $V \subset ( Z/p )^{\n} $ 
is the subspace
 (of dimension, say, $m$ ) of
 all possible $ Y_{0}$. Also $0 \notin \spec (T).$  
In other words, we have a commutative diagram:
\[
\begin{array}{ccc}
(\ldots Y_{-1}, Y_{0}, Y_{1}, \ldots) & \stackrel{\sigma}{\longmapsto}
&(\ldots Y_{0}, Y_{1}, Y_{2}, \ldots)\\
\downarrow \pi _{0} & & \downarrow  \pi _{0} \\\
Y_{0}&\stackrel {T} {\longmapsto} &Y_{1} 
\end{array}
\]

\noindent
We summarize the results of this section in the following theorem:
\begin{thm}\label{theor22}
The three dynamical systems are isomorphic:
$$ (\Hom (K, \Z /p), \sigma _x)\; \cong \;(\mathcal{V}, \sigma )\;  \cong\;
(V,\,T),$$       
where $\mathcal{V}$ is the space of bi-infinite sequences 
$ \ldots Y_{-1}, Y_0,Y_1,\ldots  $satisfying
equation 
\begin{equation}
BY_j-AY_{j+1}=0,\;\; j \in \Z;
\end{equation}
$ Y_j \in V \subset (\Z /p)^{\n} ,\;\; 
V = \pi _0 (\mathcal{V})$ and $T$ is 
 isomorphism of $V$ mapping $Y_j$ to $Y_{j+1}.$
\end{thm} 
Now we will study the spectral structure of $\sigma _x \,(\sigma ,\,T )$
 using a convenient                                    
presentation.

\section {Jordan matrix for  $\sigma  _x$}

\noindent
Throughout the rest of the paper $E $ will denote the identity operator in the 
space under consideration, and, in case of a finite dimensional space, also 
the unit 
matrix  of corresponding size.  

\subsection {Eigenvectors and adjoint vectors of the full shift}
Denote the splitting field  for Alexander polynomial over $\Z /p $ by $F.$
Let $V$ be a vector space over  $F$, 
 and let $\sigma $ 
be the left shift in the space of bi-infinite sequences $V^{Z}$.
Recall the definitions: 

\noindent
$Y^{0} \in V^{Z}$ is  {\it an eigenvector} for $ \sigma $ with eigenvalue 
$t \in F $
if $(\sigma -tE)Y^{0}=0$

\noindent
$Y^0,  Y^{1}, \ldots ,  Y^{i} $ is {\it a sequence of adjoint vectors} for 
eigenvector $Y^{0},$ if

 $(\sigma  - tE)Y^{k} = Y^{k-1}\, , \;\;k=1, \ldots , i$

\noindent Let $C_{n}^{0} = 1, \;\;\;C_{n}^{k} = 
\frac{n (n-1) \cdots (n-k+1)}{k!} \;$ 
for all integer $n$ and positive integer $k$.
\begin{prop}

\noindent a) $Y^{0}$ is an eigenvector for  $\sigma $ with an eigenvalue $t$
 if and only if $Y^{0}=\{ t^{n}v_{0}\}_{n \in \Z }$ \;\; for some
$v_{0} \in V$.          

\noindent b) $ Y^0, Y^{1}, \ldots ,  Y^{i} $ 
is a sequence of adjoint vectors for 
eigenvector $Y^{0},$
iff there  are vectors $v_{0}, v_{1}, \ldots v_{i} \in V$ such that
\begin {equation} \label{eq5}
Y^{k} = \{ C_{n}^{0}t^{n}v_{k} +  C_{n}^{1}t^{n-1}v_{k-1} + \ldots + 
 C_{n}^{k}t^{n-k}v_{0}\}_{n \in \Z }
\end {equation}
\end{prop}
\begin{proof} 
\noindent a) Let $v_{0} = Y^{0}_{0}.$
     
 $$ ((\tilde {T} - tE)\,Y^{0})_{n}\; = \;Y^{0}_{n+1} - t\,Y^{0}_{n} = 0, \;\; \;
\forall n \in Z \;\;\; \Longleftrightarrow \;\;\;
 Y^{0}_{n} = t^nY^{0}_{0}, \;\;\; 
\forall n \in Z $$  by induction.
    
\noindent b)Let $v_{k} = Y_{0}^{k}.$

The proof goes by induction by $k$, and for each $k$ the step of  induction
is from $n$ to  $n+1$ and from $\,-n$ to $\,-(n+1 \!)$.

\noindent 1)\;$((\sigma -tE)Y^{k})_{n} = Y^{k-1}_{n} \; \Longleftrightarrow \;
 Y^{k}_{n+1} - t\,Y^{k}_{n} = Y^{k-1}_{n} \;$

$\Longleftrightarrow \; Y^{k}_{n+1} = tY^{k}_{n} + 
Y^{k-1}_{n} = $ $(  C_{n}^{0}\,t^{n+1}v_{k} +  C_{n}^{1}\,t^{n}v_{k-1} + \ldots + 
 C_{n}^{k}\,t^{n-k+1}v_{0}) + $
 $ +( C_{n}^{0}\,t^{n}v_{k-1} +  C_{n}^{1}\,t^{n-1}v_{k-2} + \ldots  + 
 C_{n}^{k-1}\,t^{n-k}v_{0}) $  

\noindent and b) follows from $C_{n}^{i} +C_{n}^{i-1} = C_{n+1}^{i} .$
     
\noindent 2)\;$((T-tE)Y^{k})_{-(n+1)} = Y^{k-1}_{-(n+1)} \;
 \Longleftrightarrow \;
 Y^{k}_{-n} - t\,Y^{k}_{-n-1} = Y^{k-1}_{-n-1} \;$

$\Longleftrightarrow \; Y^{k}_{-(n+1)} = t^{-1}Y^{k}_{-n} - 
t^{-1}Y^{k-1}_{-(n+1)} =
 $ $(  C_{-n}^{0}\,t^{-n-1}v_{k} +  C_{-n}^{1}\,t^{-n-2}v_{k-1} + \ldots + 
 C_{n}^{k}\,t^{-n-k-1}v_{0}) - $
 $ -( C_{-n-1}^{0}\,t^{-n-3}v_{k-1} +  C_{-n-1}^{1}\,t^{-n-3}v_{k-2} + \ldots  + 
 C_{-n-1}^{k-1}\,t^{-n-k-1}v_{0}) $  

\noindent and b) follows from $C_{-n}^{i} +C_{-n-1}^{i-1} = C_{-n-1}^{i} $ 

\end{proof}
      
\subsection {Eigenvectors and adjoint vectors for a pencil  of matrices}
\begin{defn}
Given a pencil of matrices $ { B-tA },\;\;t \in F \; ,$  we say that 
$v_{0} $ is
{\it an eigenvector} of this pencil with eigenvalue $t_{0}  $ if
$ (B-t_{0}A) \, v_{0}=0 \; , $
and we say that $ v_{0},v_{1}, \ldots v_{i} $ is {\it a  sequence of adjoint
vectors} for eigenvector $v_{0}$ if $(B-t_{0}A) \, v_{k} = A \, v_{k-1},\;\;
 k=1, \ldots i.$
\end {defn}

\begin{prop}
The sequence \; $ Y^{0},Y^{1},\ldots Y^{i} $\; of adjoint vectors to eigenvector
$Y_{0}$ with eigenvalue $t \neq 0$\; belongs to our subshift, i.e.,  satisfies 
$$
B\,Y_{n}^{k} - A\,Y_{n+1}^k \, =\, 0 \;\;\;\forall \, n \in \Z ,\;k\,=\,0, \ldots i
$$
iff the sequence $\{ v_{0}, \ldots , v_{i} \}\,=\,
\{ \pi _{0}\,Y^{0}, \ldots \pi _{0}\,Y^{i}\}\,=\, \{Y_{0}^{0},\,Y_{0}^{1} 
\, \ldots Y_{0}^{i}\}$ is a sequence of adjoint vectors to eigenvector
 $v_{0}$ with eigenvalue $t$\;with respect to the pencil $ B\,-\,tA $.
   
 \noindent In this case $\{v_{0}, \ldots v_{i} \}\;$ is also a sequence 
of adjoint vectors to the eigenvector $v_{0}$ \; with eigenvalue  $t$ 
for the operator $T$:
   
\noindent \;$(T\,-\,t\,E)\,v_{0}\,=\,0$
  
\noindent \;$(T\,-\,t\,E)\,v_{k}\,=\,v_{k-1}$
\end {prop}
\begin {proof} is by induction by the length of sequences, $i$.
For $i=0  \;\;\; Y_{n}^{0} \,= \,t^{n}v_{0}$,\;\;\;and 
$$B \, Y_{n}^{0} \,- \,A \, Y_{n+1}^{0} \, = \,0 \;\;\;\forall \, n \in \Z 
 \;\;\;
\Longleftrightarrow \;\;\;(B \; - \,t_{0}A)\, v_{0} \,=\,0 $$ 
Suppose $\;(B \; - \,t_{0}A)\, v_{j} \,=\,A \, v_{j-1} $\;\;\;for $\,j\,=\,
1, \ldots k-1$ \; and \;$  Y_{n}^{k} \, = \, \sum_{j=0}^{k} C_{n}^{j}\,t^{n-j}\,
v_{k-j}$
Then
         
\noindent $ B \, Y_{n}^{k} \, - \, A \, Y_{n+1}^{k} \, = \, $
        
\noindent $= \,B \, \sum_{j=0}^{k} \, C_{n}^{j} \, t^{n-j} \,  v_{k-j} \, - \, 
A \,  \sum_{j=0}^{k} \, C_{n+1}^{j} \, t^{n+1-j} \,  v_{k-j} \,  = $
      
\noindent \,$ = \, B\,t^{n}\,v_{k} \,+
\,B \, \sum_{j=1}^{k} \, C_{n}^{j} \, t^{n-j} \,  v_{k-j} \,-
\,A\,t^{n+1}\,v_{k}\,-\,A \,  \sum_{j=1}^{k} \, (C_{n}^{j} \,+\,
C_{n}^{j-1})\, t^{n+1-j} \,  v_{k-j} \,  = $
$\,=\,(B-tA)\,t^{n}\,v_{k}\,+
\,(B-tA)\, \sum_{j=1}^{k} \, C_{n}^{j} \, t^{n-j} \,  v_{k-j} \,-
\,A\,\sum_{j=1}^{k} \, C_{n}^{j-1} \, t^{n+1-j} \,  v_{k-j} \,+\,$
$\,=(B-tA)\,t^{n}\,v_{k}\,+
\,A\, \sum_{j=1}^{k-1} \, C_{n}^{j} \, t^{n-j} \,  v_{k-j-1} \,-\,
A\, \sum_{s=0}^{k-1} \, C_{n}^{s} \, t^{n-s} \,  v_{k-s-1} \,=\,$
      
\noindent \; $(s=j-1)$
     
\noindent \; $=\,(B-tA)\,t^{n}\,v_{k}\,-\,A\,t^{n}\,v_{k-1}$
   
\noindent So\;\; $B\,Y_{n}^{k}\,-\,A\,Y_{n+1}^{k}\,=\,0$ \;\;\;iff\;\;\;
 $(B-tA)\,v_{k}\,=\,v_{k-1}$
  
\noindent The proof is complete.
\end {proof}

\subsection {Normal form for the Alexander matrix}
Since $\det (B\,-\, t A)$  is an Alexander polynomial,  
it  is not identically equal to $0$.  
So the pencil of matrices  $B\,-\, t A$ is strictly equivalent to its
normal quasi-diagonal  form (see \cite{Gant},Chapter XII):
 there are
 invertible matrices
 $P$ and $Q$ over $F$  
 such that
$$
B\,-\, t A \, =\, P(\tilde{B} \, -\, t \tilde{A}) \, Q ,
$$
where $\tilde{B} \, -\, t \tilde{A} $ is a  quasi-diagonal matrix
\begin {equation} \label{canon}
\{ N^{(u_{1})},N^{(u_{2})}, \ldots , N^{(u_{s})},
 J_{0} \, - \, t E, J \, - \, t E \}
\end {equation}
with $N^{(u)} \, = \, E^{(u)} \, - \, t H^{(u)}$,
    
\noindent $E^{(u)}$ is the unit matrix of dimension $u$,
      
\noindent $H^{(u)}$ is matrix of dimension $u$ with $1$'s on the first
 over-diagonal and all other elements equal to $0$,  
$J_{0}$ is a matrix in Jordan form with $0$ spectrum,
 and $J$ is a matrix in Jordan form with non-zero spectrum.
   
Since  $\tilde{B} \, -\, t \tilde{A} $ is   quasi-diagonal,
$F^{\n } \, = \, U_{1} \bigoplus U_{0} \bigoplus U $, where $ U_{1}, U_{0}, U $
 are
invariant under  $\tilde{B} \, -\, t \tilde{A} $,
 
\noindent $\left. \tilde{B} \, -\, t \tilde{A} \right| _{ U_{1}} $ has matrix
$\{ N^{(u_{1})},N^{(u_{2})}, \ldots , N^{(u_{s})} \}$
  
\noindent and
$\left. \tilde{B} \, -\, t \tilde{A} \right| _{U_{0}}  \, 
= \, J_{0} \, - \, t E,\;\;$
$\left. \tilde{B} \, -\, t \tilde{A} \right| _{U}  \, = \, J \, - \, t E.$
In particular $\left. \tilde {A} \right|_{U} \, = \, E.$

\begin{prop}
Let $ Q v_{0}  =  u_{0} \, , \;
 \ldots ,\; \; 
 Q v_{i}  =  u_{i}$

\noindent Then a) $ u_{0}$ is an eigenvector for
$ \left.\tilde{B} \, - \, t \tilde{A} \right| _{U}$ 
with eigenvalue $t_{0} \neq 0 : (\tilde{B} \, -\, t_{0} \tilde{A}) u_{0} = 0$
$\mbox{iff}\;\;
 v_{0}$ is an eigenvector for
 $B \, - \, t A $ 
with eigenvalue $ t_{0} ,\;\;\;  (B \, - \, t_{0} A) v_{0} = 0$
    
\noindent b) The sequence $ v_{0}, v_{1}, \ldots  v_{i}$
is the sequence of adjoint vectors to $v_{0}$ in the sense defined above
iff  $ u_{0}, u_{1}, \ldots  u_{i} \in U $
is  the sequence of adjoint vectors to $u_{0}$ in the regular sense.
\end{prop}
\begin{proof}
a) is true since $(B \, - \, t_{0} A) v_{0} = 
 P(\tilde{B} \, - \, t \tilde{A}) Q v_{0} = 
 P(\tilde{B} \, - \, t \tilde{A}) u_{0} $\;\;\; and $\det P \neq 0.$
  
\noindent b)  $(B \, - \, t_{0} A) v_{k} = A v_{k-1} \Longleftrightarrow $
\[ \Longleftrightarrow 
P(\tilde{B} \, - \, t_{0} \tilde{A}) Q v_{k} = P \tilde{A} Q v_{k-1} 
\Longleftrightarrow (\tilde{B} \, - \, t_{0} \tilde{A}) u_{k} = 
\tilde{A} u_{k-1}
\]
By induction, we assume that $ u_{k-1} \in U $, then $ \tilde{A} 
u_{k-1} = u_{k-1}$ and $ u_{k} \in U $,  since $U$ is invariant under
$ \tilde{B} \, - \, t \tilde{A}$.
\end{proof}
\begin {thm}\label{theor2}
 Matrix $J$ in the canonical form of the  Alexander matrix defined above
 is the Jordan matrix for operator $T$, and so for $\sigma _{x}$. 
\end {thm}
\begin {proof}

We have one-to-one correspondence between sequences of adjoint vectors 
$u_{0},u_{1}, \ldots u_{k}$ for $J$ and $v_{0}, v_{1}, \ldots v_{k}$ for $T$.
To each maximal sequence corresponds a Jordan block in that basis.
\end{proof}

In the following  section we remind to the reader how to reconstruct 
the Jordan matrix $J$ 
if we know the invariant polynomials for  Alexander matrix.

\section {Invariant polynomials of the Alexander matrix}

Recall that for a polynomial matrix $ m \times n,$ 
$D_{k}(t),\; k=1, \ldots n ,$ denotes the maximal common divisor of all
minors of order $k$ . So for $\mathcal{A}(t), \;\;
D_{n}(t) = \det \mathcal {A}(t)$ is the Alexander polynomial of the knot.
These polynomials are the same for original matrix $\mathcal{A}(t)$ and 
 its normal form (\ref{canon}). 
Up to multiplication by $t^{s}, s \in \Z ,$  
they are the same as polynomials for matrix
$J-tE. $ 
Note that $ D_{k+1}$ is divisible by $D_{k}$ .
The  invariant polynomials for a polynomial matrix are defined as 
\[
\begin{array}{c}
 i_{1}(t) = \frac {D_{n}(t)}{D_{n-1}(t)} \\
\vdots \\
 i_{k}(t) = \frac {D_{n-k+1}(t)}{D_{n-k}(t)}
\end{array}
\]

It is easy to see that invariant polynomials for matrix $J - t E$, where $J$ 
is  a Jordan matrix, are:
\begin{equation}
 i_{1}(t) = \frac {D_{n}(t)}{D_{n-1}(t)} = (\lambda_{1} - t)^{a_{1}}\;
(\lambda_{2} - t)^{b_{1}} \ldots  (\lambda_{k} - t)^{s_{1}}
\end{equation}
where $\lambda _{1},\;\lambda _{2}, \ldots \lambda _{k} $ are all distinct
eigenvalues of $J$,

\noindent $a_{1}$ is the size of maximal Jordan cell with
diagonal element $\lambda _{1}, $ 

\noindent $b_{1}$ is the size of maximal Jordan
cell with diagonal element $\lambda _{2},$

$\vdots$

\noindent $s_{1}$ is the size of maximal Jordan cell with
diagonal element $\lambda _{k}, $ 
( since for a Jordan cell $(( \lambda - t) E + I_{1}), 
D_{k} = ( \lambda - t)^{k},$ and  $D_{k-1} = 1 $).
Similarly
\begin {equation}
 \frac {D_{n}(t)}{D_{n-2}(t)} = (\lambda_{1} - t)^{a_{1}+a_{2}}\;
(\lambda_{2} - t)^{b_{1}+b_{2}} \ldots  (\lambda_{k} - t)^{s_{1}+s_{2}}
\end {equation}
where $a_{2},\; b_{2}, \ldots  s_{2}$ are the second maximal sizes 
of corresponding Jordan cells.
So
\begin {equation}
 i_{2}(t) = \frac {D_{n-1}(t)}{D_{n-2}(t)} = (\lambda_{1} - t)^{a_{2}}\;
(\lambda_{2} - t)^{b_{2}} \ldots  (\lambda_{k} - t)^{s_{2}}
\end {equation}
So,
to reconstruct matrix $J$ 
 from invariant
polynomials of $\mathcal{A}(t)$, we add a Jordan cell of size $k$ with diagonal
element $\lambda $ for each root $\lambda \neq 0 $ of multiplicity $k$ for each
invariant polynomial.

\section {The least common multiple of periods of orbits 
of $(\Hom(K, \Z  /p) ,  \sigma  _x )$}

The following theorem is an immediate corollary of Theorem (\ref{theor2}).
\begin{thm}
The least common multiple of periods of orbits of dynamical system 
$(\Hom(K, \Z  /p),  \sigma  _x)$ is equal to the order of $J$ 
as an element of  finite group $GL(m,F)$, i.e. to the minimal integer $N > 0$
such that $J^N = E ;$ where $J$ is a Jordan matrix with non-zero spectrum
from the normal form of Alexander matrix.
\end{thm}

\begin{lem}\label{lemma}
Let $n \;$ be any integer number and $k\;$ be an integer such that 
$k > 1.$

\noindent
$ C_{n}^{1}, \ldots  C_{n}^{k-1}$ are divisible by $ p $ iff $n$ is divisible
by $p^r \;,$ where  $p^{r-1} < k \leq  p^{r}.$
\end {lem}
\begin{proof}
If $k=2, \;$ then $r=1 \;$ and 
 $\;C^1_n $ is divisible by $p \;
\Longleftrightarrow \;n \; $ is divisible by $ p.$ 
We state that $n$ is divisible by $p^{s-1}$ and   
 $ C_{n}^{p^{s-1}}$ is divisible by $p$ iff $n$ is divisible by $p^{s}.$
Indeed,
$
C_{n}^{p^{s-1}} = \frac {n(n-1) \cdots (n-p^{s-1}+1)}{1 \cdots p^{s-1}}
$,
and among $p^{s-1}$ factors in the numerator there must be one divisible
by $p^{s}$ (otherwise the sum of exponents of $p$ in the numerator and
 denominator is the same), and it is $n$ , since $n$ is the only factor
divisible by $p^{s-1}.$ 
So $C_{n}^{1}, C_{n}^{2}, \ldots C_{n}^{k-1}$
 are divisible by $p$ and $k > p^{r-1} $ 
(which is equivalent to $k-1 \geq p^{r-1}$) implies  $n$ is divisible by 
$p^{r}.$ 
 And opposite, $n$ is divisible by $p^{r} \mbox{and} \;  k \leq  p^{r}\;$
(which is equivalent to $k-1 < p^r $)
implies $\;  C_{n}^{1}, C_{n}^{2}, \ldots C_{n}^{k-1}$ are divisible
by $p$ . 
\end{proof}

\begin {thm}\label{order}

Let $J$  be a matrix in Jordan form: 
$$
J= \{ \lambda _{1}E^{(k_{1})}+H^{(k_{1})}, \; \ldots \;
 \lambda _{l}E^{(k_{l})}+H^{(k_{l})} \}
$$ 
where $\lambda _{i} \neq 0$
are not necessary distinct; $ \lambda _{i}$ belong to a field $F$ 
of characteristic $p$.

\noindent
For each Jordan cell with eigenvalue $\lambda _{i} $ 
take 
 $n_{i} = lcm (d_{i},p^{r_{i}})$ 
where $d_{i}$ is the order of $\lambda _{i}$ and 
$p^{r_{i}-1} < k_{i} \leq p^{r_{i}}$ . Then the order of $J$ 
is $lcm(n_{i}|1 \leq i \leq l).$
\end {thm}
\begin {proof}
 
Let $J$ be a Jordan cell  of dimension $k$ with  
eigenvalue $\lambda$, and let $d$ be the order of $ \lambda$. 
If $k=1 $ so that $ p^{-1} < k \leq 1 ,$ then $ J^{n} = E$ 
 iff $n$
is divisible by $d$.
     
\noindent Suppose $k > 1$ and $p^{r-1} < k \leq p^r,$
For $i \leq k-1 $ denote by $I_{i}$ the matrix with $\bar {1}$  
on the $i -$th over-diagonal
and $\bar {0}$ everywhere else. And let $I_{k} = I_{k+1} = 
\ldots = 0.$ Then $I_{1}^{n} = I_{n} \; $ and
$$
\noindent J^{n} = (\lambda E + I_{1})^{n} = 
\lambda ^{n} E + \lambda ^{n-1} C_{n}^{1} I_{1} + 
\lambda ^{n-2} C_{n}^{2} I_{2} + \ldots +
 C^n_n I_n.
$$
If $n \leq k-1, \; J^n \neq E \;$ because of the term $I_n .$
If $n > k-1,$
$$
\noindent J^{n} = (\lambda E + I_{1})^{n} = 
\lambda ^{n} E + \lambda ^{n-1} C_{n}^{1} I_{1} + 
\lambda ^{n-2} C_{n}^{2} I_{2} + \ldots +
\lambda ^{n-(k-1)} C_{n}^{k-1} I_{k-1}.
$$
Then 
$ J^{n} = E \Longleftrightarrow \lambda ^{n} = \bar {1}$ and
$ C_{n}^{1}, \ldots  C_{n}^{k-1}$ are divisible by $ p .$
By lemma \ref{lemma} minimal such $n$ is $lcm(d, p^r).$
\end{proof}

\section {The periods of adjoint vectors}

\begin{prop}\label{adjvec}
Let $\tilde{V}$ be a vector space over $F , \;
T  : \tilde{V}  \longrightarrow \tilde{V} \;$
 be an isomorphism  with Jordan
matrix $\;J, \; J \in GL(m,F).$ 
Consider the Jordan basis for $T$ in $\tilde{V}.$
 If a basis vector 
$e^s $ is $s$-th adjoint vector to a basis eigenvector $e^0 $ 
 with eigenvalue $\lambda ,$ i.e. 
\begin{equation}\label{adjeq}
(T - \lambda E)\,e^0=0 \; ,\;\; (T - \lambda E)\,e^i =e^{i-1} \;,\;
 i=1,\ldots s,
\end{equation}
then $e^s $ is periodic under $T$ with period $\;N=lcm(d,p^r),\;$
where $d$ is the order $\mbox{of} \;\lambda , \;\;\mbox{and} \;\;
 p^{r-1} < s+1 \leq p^r.$
\end{prop}
\begin{proof}
We have $\;T^n e^0 = \lambda ^n e^0.\;$ 
So $e^0$ has period $d $ under $T.$
If $s \geq 1,$ then
$T e^s = \lambda e^s + e^{s-1} \;\; \mbox{and, by induction,} $
\begin{equation}\label{Tnes}
T^n e^s = \lambda ^n e^s + C^1_n \lambda ^{n-1}e^{s-1}+ \cdots
 + C^s_n \lambda ^{n-s}e^0,
\end{equation}
and the result follows from Lemma \ref{lemma}.
\end{proof} 

\begin{prop}\label{ddiv}
Let $\tilde{V},\; T $ and $J $ be as in proposition \ref{adjvec}. 
Suppose $J$ has $l$ Jordan cells
with eigenvalues $\lambda _1, \ldots ,\lambda _l$ (not necessary distinct),
of orders $d_1, \ldots , d_l $ respectively,
so that we can write 
$\tilde{V}=V_1 \oplus V_2 \oplus \ldots \oplus V_l$
where $V_i $ are the invariant subspaces corresponding to the Jordan cells. 
If the projection of a vector $Z \in \tilde{V} $ to $V_i$ is not zero, then the
period of $Z$ under $T$ is divisible by $d_i.$
\end{prop} 
\begin{proof}
The projection of vector $Z$ onto $V_i$ 
is a linear combination of basis vectors
belonging to $V_i :
Z_i = \alpha _0 e^0 + \alpha _1e^1 + \ldots + \alpha _s e^s,$ where  
$e^0, \ldots e^s$ are the  adjoint vectors to the eigenvector 
$e^0$ with eigenvalue $\lambda _i.$ 
Suppose that the highest non-zero coefficient is $\alpha _{\tau}.$
Then $T^n Z $ will have the coefficient $\lambda _i ^n \alpha _{\tau } $ 
at $e^{\tau}$ by formula (\ref{adjeq}) and the result follows.
\end{proof} 

Let us consider the extensions of our linear spaces 
$\Hom (K, \Z /p),\; \mathcal{V} ,\;  V ,$ and linear operators 
$\sigma _x, \sigma ,\; T$ to the field $F.$ Denote the resulting dynamical
systems by $(\Hom (K,F), \sigma _x), (\mathcal{F}, \sigma ), (\tilde{V},T).$
Here $\mathcal{F} $ is the space of bi-infinite sequences $\; (\ldots , Z_{-1},
Z_0, Z_1 \ldots) ,\; z_j \in F^{\n}$ satisfying equation (\ref{rec0}), 
 $\tilde{V}= \pi _0 (\mathcal{F}) $ is $m- $dimensional linear vector space
over $ F $.
Applying propositions \ref{adjvec} and \ref{ddiv}
to the extended operator $T: \tilde{V} \longrightarrow \tilde {V}$ 
we get the following result:
\begin{thm}\label{finalgen}
Let $\lambda _j ,\; j=1,\ldots , l ,$ 
be the (non-zero) roots of the Alexander polynomial of the knot,
$d_j$ be the order of $\lambda _j ,\;$
$k_j$ be the maximal size of a Jordan cell of the matrix $J$ 
with eigenvalue $\lambda _j \;$ $ \mbox{(see section 6)}.$ Let 
 $p^{r_j-1} < k_j \leq p^{r_j}.$ For each $\lambda _j $ 
consider the set $Q_j = \{1,\, lcm(d_j, p^i):  
i=0, \ldots r_j \}.$ Then the set of periods of orbits of 
extended dynamical system
$(\Hom(K,F),\sigma _x) $
 is a subset of the set 
$$
\{lcm(q_1,\ldots , q_l): q_j \in Q_j, j=1,\ldots ,l \}.
$$ 
\end{thm}

\section{ Two-bridge knots }

Let $\mathcal{K}$ be a 2-bridge knot,
 and  $\; \mathcal{A}(t) = c_{0}+c_{1}t+ \cdots +c_{m}t^{m} \;$
 be its Alexander polynomial
over $ \Z /p .$
Theorem \ref{theor22}  gives
the the following corollary: 
\begin{thm}\label{prop10.1}

\noindent
\noindent
i) $( \Hom (K, \Z /p), \, \sigma _{x})$ is isomorphic to the linear space of 
bi-infinite sequences
$$ \ldots y_{-1}, y_{0}, y_{1} \ldots \;\;\;\;, \; y_{j} \in \Z /p ,$$
satisfying
\begin{equation}\label{rel1}
 c_{0}y_{j}+c_{1}y_{j+1}+ \cdots +c_{m}y_{j+m} \;= 0,
\end{equation}
with the left shift ( a linear operator ) $\sigma . $

\noindent
ii) $( \Hom (K, \Z /p), \, \sigma _{x}))$ is isomorphic to $(\Z /p)^{m}$
with the linear operator $T$ with matrix 
\[
\left (
\begin{matrix}
0 & 1& &&&\\
&0&1&&& \\
&&\ddots&&\\
\\
&&&&0&1\\
- \frac{c_{0}}{c_{m}} & - \frac{c_{1}}{c_{m}}&& \ldots && - \frac{c_{m-1}}{c_{m}}
\end{matrix}
\right )
\]

\noindent
iii) The Jordan matrix for $\sigma _x (\sigma ,  T )$  has exactly one cell
 for each root of
Alexander polynomial $\mathcal{A} (t).$

\end{thm}

\begin{proof}

\noindent
i) follows from theorem \ref{theor1}, if we choose a presentation with two
generators and one relation for $G$.

\noindent
ii) is obvious if we take  the linear operator $T: (\Z /p)^{m}
 \rightarrow (\Z /p)^{m}$
\[ 
T \left ( 
\begin{array}{l}
y_{0}\\
y_{1}\\
\vdots\\
y_{m-1}
\end{array} 
\right )
=
\left ( 
\begin{array}{l}
y_{1}\\
y_{2}\\
\vdots\\
y_{m}
\end{array} 
\right )
\]

\noindent
iii) We have \[\Det (T-tE) = \frac{c_{0}}{c_{m}} + \frac{c_{1}}{c_{m}}t+ 
\ldots \frac{c_{m-1}}{c_{m}}t^{m-1}+t^{m}= \frac{1}{c_{m}} \mathcal{A} (t),
\]
and the greatest common divisor of all minors of order $m-1$ of matrix $T$
is $1$. It follows
that in the Jordan form for $T$ for each root of $\mathcal{A}(t)$
there is exactly one cell of the size equal to the multiplicity of the root.
\end{proof}

Suppose all roots of $\mathcal{A} (t)$ belong to the same orbit of Galois
group action. Denote them by $\; \lambda _1, \ldots , \lambda _l .$ 
 Then they all
have the same multiplicity $k$ and order $d$. Let
$ p^{r-1}<k \leq p^r.$

\begin{thm}
There are orbits of $( \Hom (K, \Z /p ), \sigma _x) $ of periods 
$lcm (d,p^i) $ for all $\; i $ such that $ 0 \leq i \leq r.$
\end{thm}

The proof will follow from propositions \ref{jordanbasis} and 
\ref{integerseq}.

\begin{prop}\label{jordanbasis}
The space $\mathcal{F} $  of all sequences
$ (\ldots z_{-1}  , z_0 , z_1 \ldots ),  \; z \in F,$
satisfying $$c_0z_j+c_1z_{j+1} + \cdots + c_mz_{j+m} = 0 $$
with the left shift $\sigma$  has the Jordan basis 
$ e^0_{\lambda _1} \ldots e^{k-1}_{\lambda _1}, \ldots , \;
 e^0_{\lambda _l} \ldots e^{k-1}_{\lambda _l },\;$
 where $e^i_{\lambda} \; $is a bi-infinite sequence
$
e^i_{\lambda} = \{ e^i_{\lambda  j}\}\, _{j \in \Z }\; :$
$$
\;e^0_{\lambda j}= \lambda ^j,\;
e^1_{\lambda j}= \frac{\partial}{\partial \lambda} \lambda ^j
\;\;,  \ldots \
;,\;
\;e^{k-1}_{\lambda j}=\frac{1}{(k-1)!} 
\frac{\partial^{k-1}}{\partial \lambda ^{k-1}} \lambda ^j \;
\mbox{ for} \;\lambda = \lambda _1, \ldots, \lambda _l\;.
$$
\end{prop}
\begin{proof}
First we prove that these sequences form a basis.
The dimension of $\mathcal{F}$ is equal to the degree of the Alexander 
polynomial, 
so we have right number of vectors. To prove their linear independence, 
it is enough to prove that the  matrix $Q$ whose columns are 
formed by segments of length $m$ of the sequences above is non-degenerate. 
More precisely, $Q$ is obtained from $l$ blocks, $l$ being the number of roots
of $\mathcal{A}(t), $ $i$ numerates columns, $j$ numerates rows   
\begin{equation} \label{Q}
Q\;=\;
\left [ e^{\,i}_{ \lambda _1\, j} 
\right ]
\cdots
\left [ e^{\,i}_{\lambda _l \,j}
\right ]
\mbox{\;\;\;where\;}
\left \{
  \begin{array}{l}
i=0, \ldots k-1
\\j=0, \ldots lk-1
\end{array}
\right.
\end{equation}

\noindent
Suppose $Q$ is degenerate. Then there is a linear combination of rows
with coefficients, say, $b_0, \ldots b_{lk-1} ,$ that is equal to zero.
For each $\lambda = \lambda _1,\ldots ,\lambda _l $     we have 
$$
b_0 +b_1 \lambda + \cdots +b_{lk-1} \lambda ^{lk-1}=0,\;\;
\sum_{j=0}^{lk-1} b_j \frac{1}{i!} \frac{\partial ^i}{\partial \lambda ^i }
\lambda ^j \;=0, \;\;\mbox{for}\; i=1 \ldots k-1,
$$
 But then polynomial 
$p(t)= b_0+b_1t+ \ldots b_{lk-1}t^{lk-1} $ and the polynomials 
$\frac{1}{i!}\frac{\partial ^i}{\partial t^i}p(t), \;\;i=1,\ldots k-1$
all have roots $\lambda _1, \ldots , \lambda _l \,,$ 
i.e. $\lambda _1  \ldots \lambda _l \;$
are the roots of $p(t)\,,$ each of multiplicity $k\;,$ while the degree 
 of $p(t) $ is $lk-1.$ Contradiction.

Now we'll prove that the sequence $e^0_{\lambda },\ldots , e^{k-1}_{\lambda } $
is a sequence of adjoint vectors for each $\lambda   = \lambda _1, \ldots ,
\lambda _l .$ Indeed,
$\;e^0_{\lambda }= \{ \lambda ^j \}\: _{j \in \Z} \; $ 
is an eigenvector
for $\sigma $ with eigenvalue $\lambda,  \;$ since
$$ (\sigma - \lambda E)\, e^0_{\lambda} =  \{ \lambda ^{j+1} \}\: _{j \in \Z}-
\lambda \{ \lambda ^j \}\: _{j \in \Z}\; =\, \{0 \}_{j \in \Z }\;\;\;(\;E\;\; 
\mbox{is an identity operator}).
$$
And $e^i_{\lambda}$ is its $i\,$-th adjoint vector:
\begin{equation}\label{adj}
 (\sigma  - \lambda E) e^i_{\lambda}=e^{i-1}_{\lambda}.
\end{equation}
Indeed, note  that  for polynomials $\;P( \lambda ) ,\;  Q(\lambda ,)$
$$
\frac{\partial^n}{\partial \lambda ^n} (P(\lambda )Q(\lambda)) = 
\sum_{i=0}^{n} C_n^i \frac{\partial^i}{\partial \lambda ^i}P(\lambda )\,
\frac{\partial^{n-i}}{\partial \lambda ^{n-i}}Q(\lambda ).
$$
So
$$
e^i_{\lambda \,j+1}-\lambda \, e^i_{\lambda j} \;
 = \;\frac{1}{i!}\, \frac{\partial^{i}}{\partial \lambda ^{i}} 
(\lambda  ^j \lambda )- \lambda\, 
 \frac{1}{i!}\, \frac{\partial^{i}}{\partial \lambda ^{i}} 
(\lambda  ^j)\;=
$$
$$
=\;
\frac{1}{i!}\, (\frac{\partial^{i}}{\partial \lambda ^{i}} 
\lambda  ^j) \, \lambda + \frac{1}{i!}\,C^1_i\, (\frac{\partial^{i-1}}
{\partial \lambda ^{i-1}}\lambda ^j)     \,\frac{\partial}{\partial \lambda } 
\lambda -\lambda \,\frac{1}{i!}\, \frac{\partial^{i}}{\partial \lambda ^{i}} 
\lambda  ^j \;=
$$
$$
 =
 \frac{1}{(i-1)!}\,
 \frac{\partial^{i-1}}{\partial \lambda ^{i-1}} 
\lambda  ^j \;= \;e^{i-1}_{\lambda j}  
 $$
\end{proof} 

Let us  call the sequences $ (\ldots , y_{-1},\, y_0,\, y_1 \ldots )\;\; 
\in \mathcal{V} \subset \mathcal {F} \;\;\mbox{with}\;\; y_j \in \Z /p  \;\;
 ${\it integer sequences }.

\begin{prop}\label{integerseq}
For any $ s $ such that $0 \leq s \leq k-1 ,$ 
the sequence $e^s_{\lambda _1}+
\ldots +e^s_{\lambda _l}$ is integer and has period equal to lcm$(d,p^r) $,
 where  $p^{r-1} < s+1 \leq p^r.$
\end{prop}
 \begin{proof}
The sequence $e^s_{\lambda _1} +\ldots  +e^s_{\lambda _l}$ 
is integer since it is invariant 
under the action of Galois group.
 
We know from proposition \ref{adjvec} that $e^s_{\lambda _1}, \ldots , 
e^s_{\lambda _l } $
all have period $lcm(d,p^r). $  Their sum has the same period since
 $e^s_{\lambda _1}, \ldots , 
e^s_{\lambda _l } $ belong to the different invariant subspaces 
in the direct sum
$\mathcal{F}=U_1 \oplus \cdots \oplus U_l.$
\end{proof} 

In case of $l $  orbits of the Galois group action 
the Alexander polynomial (characteristic polynomial of operators $\sigma _x,
\; \sigma ,\;T) $
 is a product of $l$ mutually prime polynomials:
$\mathcal{A}(t)=f_1(t) \cdots f_l(t),$ and the space, say, $\tilde {V} $ 
is the direct sum of subspaces invariant under $T\;$ :
$\tilde {V} = U_1 \oplus \ldots \oplus U_l.\; $
Taking  sums of orbits from different subspaces
we obtain the final result for two-bridge knots:

\begin{thm}\label{finalthm}
Let $\lambda _j, j=1\ldots  l,$ be the (non-zero) roots of Alexander 
polynomial of a  two-bridge knot, of multiplicity $k_j$ and order $d_j$ 
respectively.
  Let $r_j $ be the integer numbers such that $p^{r_j-1}<k_j \leq p^{r_j}$.
 Denote the set 
$\{1, lcm(d_j, p^i): 0 \leq i \leq r_j \} $ by $Q_j $ for each root 
$\lambda _j $ (for conjugate roots these sets coincide). 
Then the set of periods of orbits of $( \Hom (K,\Z /p), \sigma _x)\; $ is 
the set
  $$ \{lcm(q_1,\ldots , q_l): \; q_j \in Q_j ,\; j=1, \ldots l \}.$$
\end{thm}

\section {Examples}
\begin{example}
Trefoil is a two-bridge knot with Alexander polynomial $ 1-t+t^2\,
\equiv \,(1+t)^2 \,$ in $\Z /3.$
It has one root, -1, of order 2 and multiplicity 2.
So $d=k=2,\; r=1 $ and the least common multiple of periods is $lcm(2,3) = 6.$
All non trivial orbits have period $2 $ or $6.$
The recurrent equation is $\,y_j -y_{j+1} +y_{j+2} \equiv 0 \,(\mbox{mod}3).$  
The orbits of $\sigma $ of period $2$ are given by sequence
$\;( \ldots  1,-\!1, 1, -\!1, \ldots) $
and the orbits of period $6$ are given by sequence 
$\;\; (\ldots 0, 1, 1, 0, -\!1, -\!1,  \; \ldots \;).$
\end{example}

\begin{example}
The figure 8 knot is a two-bridge knot with Alexander polynomial

\noindent
$-1+3t-t^2 \,\equiv  \, -1-t^2 \,$ in $\Z /3.$ 
It has two simple roots of order 4. 
So all non-trivial orbits have period 4.
The recurrent equation is $y_j +y_{j+2} =0 $
and  non-trivial orbits are given by sequences $(\ldots ,1,1,-1,-1,\ldots )$
and $(\ldots,1 ,0,-1,0, \ldots) .$
\end{example}

\begin{example}
For two-bridge knot $\mathcal{K} = 9_{1}$  the Alexander
polynomial is $ \mathcal{A}(t)=
(t^{2}-t+1)(t^{6}-t^{3}+1) \equiv (t+1)^{8} $ in  $\Z /3 .$
There is one root, $-1$, of order $2$ and  multiplicity $8$.  
So $ l=1, d=2, k=8, r=2,  $ and all  non-trivial orbits have periods 
2,\; 6 or  18. The recurrent equation is 
$$y_j-y_{j+1}+y_{j+2}-y_{j+3}+y_{j+4}-y_{j+5}+y_{j+6}-y_{j+7}+y_{j+8}=0 .$$
 Examples  are given by sequences:

\noindent
$\ldots , 1,-1, \ldots \;\;$ (period 2)

\noindent
$\ldots , 0,0,1,0,0,-1, \ldots \;\;$ (period 6)

\noindent
$\ldots,  0,0,0,0,0,0,0,1,1,0, 0, 0, 0, 0, 0, 0, -1,-1 \ldots \;\;$ (period 18).
\end{example}

\begin{example}
The knot $6_{2}$ has the Alexander polynomial $t^{4}-3t^{3}+3t^{2}-3t+1 \equiv 
t^{4}+1 \;\; \mbox{over} \;\; \Z /3, $ which has four simple roots of order $8$.
So  all its non-trivial orbits have period 8.
\end{example}


\begin{thebibliography}{*****}

\bibitem [CF]{CrF} R.H. Crowell, R.H. Fox, Introduction to Knot Theory,
Ginn and Co., 1963
\bibitem [G]{Gant} F.R. Gantmacher, The Theory of Matrices,
Chelsea Publishing Co., 1959
\bibitem [M]{M} K. Murasugi, Knot Theory and Its Applications, Birkh\"auser, 
1996 
\bibitem [S]{Si} D.S. Silver, ``Augmented group systems and $n-$ knots,''
Math.Ann. \boldmath${296}$ (1993), 585-593. MR \boldmath${94i}$:57039
\bibitem [SW1]{SiWi 1} D.S. Silver and S.G. Williams,  
 ``Augmented group systems and shifts of finite type,''
Israel J. Math. \boldmath${95}$, (1996), 231-251. MR \boldmath${98b}$:20045  
\bibitem [SW2]{SiWi 2} D.S. Silver and S.G. Williams, 
 ``Knot invariants from symbolic dynamical systems,''
Trans.Amer.Math.Soc. V351 N8 (1999), p3243-3265,S 0002-9947(99)02167-4
 

\end{thebibliography}
\end{document}